\documentclass[11pt, notitlepage]{article}   

\usepackage[title]{appendix}

\usepackage{amsmath,amsthm,amsfonts,hyperref}   

\usepackage{bbm} 
\usepackage{amscd, array}
\usepackage{amsfonts}
\usepackage{amssymb}
\usepackage{color}      
\usepackage{epsfig}
\usepackage{graphicx}           
\usepackage{algorithm}
\usepackage[noend]{algpseudocode}

\usepackage{tikz}
\usetikzlibrary{shapes}
\usetikzlibrary{matrix}
\usetikzlibrary{positioning}
\usetikzlibrary{fit}

\tikzstyle{res}=[circle,thick,minimum size=4mm,draw=black,fill=red,inner sep=1pt]
\tikzstyle{non-res}=[circle,thick,minimum size=4mm,draw=black,inner sep=1pt]
\tikzstyle{light-res}=[circle,thick,minimum size=4mm,draw=black,fill=red!40,inner sep=1pt]
\tikzstyle{blue}=[circle,thick,minimum size=4mm,draw=black,fill=blue!20,inner sep=1pt]

\newtheorem{theorem}{Theorem}[section]

\newcommand{\mo}{\textnormal{Mo}}
\newcommand{\irr}{\textnormal{irr}}
\newcommand{\peri}{\textnormal{peri}}
\newcommand{\spr}{\textnormal{spr}}
\newcommand{\eperi}{\textnormal{eperi}}
\newcommand{\espr}{\textnormal{espr}}
\theoremstyle{remark}
\newtheorem*{remark}{Remark}
\theoremstyle{theorem}

\newtheorem{claim}[theorem]{Claim}
\newtheorem{definition}[theorem]{Definition}
\newtheorem{lemma}[theorem]{Lemma}
\newtheorem{thm}[theorem]{Theorem}
\newtheorem{conjecture}[theorem]{Conjecture}

\def\finf{\mathop{{\rm I}\kern -.27 em {\rm F}}\nolimits}

\usepackage[colorinlistoftodos,textsize=tiny]{todonotes}
\newcommand{\Comments}{1}
\newcommand{\mynote}[2]{\ifnum\Comments=1\textcolor{#1}{#2}\fi}
\newcommand{\mytodo}[2]{\ifnum\Comments=1%
  \todo[linecolor=#1!80!black,backgroundcolor=#1,bordercolor=#1!80!black]{#2}\fi}

\begin{document}

\title{Extremal Bounds on Peripherality Measures}
\author{Linus Tang}
\date{June 18, 2023}

\maketitle

\begin{abstract}
\noindent
We investigate several measures of peripherality for vertices and edges in networks. We improve asymptotic bounds on the maximum value achieved by edge peripherality, edge sum peripherality, and the Trinajstić index over $n$ vertex graphs. We also prove similar results on the maxima over $n$-vertex bipartite graphs, trees, and graphs with a fixed diameter. Finally, we refute two conjectures of Furtula, the first on necessary conditions for minimizing the Trinajstić index and the second about maximizing the Trinajstić index.
\end{abstract}



\section{Introduction}

A centrality measure is an approximate descriptor of the importance or influence of a vertex in a graph. Central vertices are often ones with large degrees and are critical in the structure of the graph. Peripherality measures are the opposite; peripheral vertices are the least important in a graph. Centrality and peripherality measures can also be defined on edges of a graph.

Measures of centrality and peripherality can be used to analyze systems of chemical reactions in order to determine the most and least important reactions and chemical species in a system. Silva et al.\ \cite{Chem2} analyzed three systems of atmospheric chemical reactions called SuperFast, GEOS-Chem v12.6, and the Master Chemical Mechanism v3.3 through a directed graph model in which the vertex set consisted of chemical species and reactions. Edges were from reactions to their products and to reactions from their reactants. In order to determine the most important chemical species within each system, the authors computed a certain centrality measure, called out-degree centrality, for each vertex. Geneson and Tsai \cite{gt22} defined undirected graphs on atmospheric chemical reaction systems SuperFast and $\text{MOZART}$-4 and analyzed various measures of centrality and peripherality for the vertices and edges to determine the most and least important reactants and reactions within each system.

Asif et al.\ \cite{Neural} studied the applications of topological indices, especially the Mostar index, to neural networks. Because peripheral vertices are placed in less influential positions in the graph, they do not affect the flow of information as quickly. Thus, the authors posit that the flow of data in probabilistic neural networks (PNN) and convolutional neural networks (CNN) is optimized when their peripheralities are minimized. Asif et al.\ compute the Mostar index (a measure of peripherality of a graph) for PNN and CNN.

Geneson and Tsai \cite{gt22} defined measures called peripherality, edge peripherality, and edge sum peripherality, denoted $\peri(v)$, $\eperi(\{u,v\})$, and $\espr(\{u,v\})$, respectively. They also define corresponding measures for graphs $G$, denoted $\peri(G)$, $\eperi(G)$, and $\espr(G)$. They computed the values of $\peri(G)$ for families of graphs including complete graphs, cycles, paths, and complete bipartite graphs. They also proved extremal results including the minimum and maximum values of $\peri(G)$ over connected graphs $G$ of order $n$. They computed the values of $\eperi(G)$ for families of graphs including complete graphs, cycles, paths, and complete bipartite graphs. They also asymptotically determined maximum values of $\eperi(G)$ over connected graphs $G$ of order $n$. They computed the values of $\espr(G)$ for families of graphs including complete graphs and complete bipartite graphs. They also asymptotically determined maximum values of $\espr(G)$ over connected graphs $G$ of order $n$. 

For a pair $\{u,v\}$ of distinct vertices, Do\v{s}li\'{c} et al.\ \cite{Mo3} defined its Mostar index, denoted $\mo(\{u,v\})$. They also defined the Mostar index of a graph $G$, denoted $\mo(G)$. A related index, defined by Miklavi\v{c} and \v{S}parl \cite{MikTotalMo}, is the total Mostar index, or distance-unbalancedness, denoted $\mo^*(G)$. The Mostar index has been the subject of dozens of papers since it was introduced by Do\v{s}li\'{c} et al.\ in \cite{Mo3}. Deng and Li \cite{MoChem} proved extremal results about the Mostar index of chemical trees (trees where every vertex has degree at most 4) and general trees of any fixed diameter. In \cite{MoDegSeq}, the same authors determined the trees of any fixed degree sequence that maximize the Mostar index. In \cite{ghorbani}, Ghorbani et al.\ determined a sufficient condition for the Mostar index of a graph to equal 0. In \cite{Mo2} and \cite{Mo1}, Miklavi\v{c} et al.\ proved bounds on the maximum Mostar index over connected graphs, split graphs, and bipartite graphs. In \cite{KrRa} and \cite{KrRa1}, Kramer and Rautenbach established bounds on the maximum and minimum distance-unbalancedness of trees.

Furtula \cite{NT} defined the Trinajstić index of a graph, denoted $NT(G)$. Furtula \cite{NT} ran statistical tests to find the correlation between the Trinajstić index and several other measures, finding a loose correlation with the total Mostar index. Few extremal results about the Trinajstić index have been proven, but computer searches led Furtula to three conjectures on trees that minimize $NT(G)$, graphs that maximize $NT(G)$, and graphs that minimize $NT(G)$.

For an edge $\{u,v\}$ of a graph, Albertson \cite{irr} defined its irregularity, denoted $\irr(\{u,v\})$. Albertson also defined the irregularity of a graph $G$, denoted $\irr(G)$. Gao et al.\ \cite{Mo-irr} determined the maximum difference $\mo(T)-\irr(T)$ over $n$-vertex trees $T$ for all $n\leq22$. Geneson and Tsai \cite{gt22} proved as asymptotic bound of $n^2(1-o(1))$.

In Section \ref{eperi}, we improve the extremal bounds on the maximum value attained by $\eperi(G)$ over $n$-vertex graphs. Whereas Geneson and Tsai proved in \cite{gt22} that this value was between $\frac2{125}n^3$ and $\frac12n^3$, we narrow these bounds to $\frac{\sqrt3}{24}n^3(1-o(1))$ and $\frac16n^3$. In Section \ref{espr}, we improve the analogous bounds on $\espr(G)$ from $\frac18n^4-O(n^2)$ and $n^4$ to $\frac5{32}n^4-O(n^3)$ and $\frac14n^4$, respectively. Additionally, in \ref{esprdiam2}, we prove that the maximum value of $\espr(G)$ over $n$-vertex graphs $G$ with diameter 2 is $\frac4{27}n^4-O(n^3)$. We also prove that the maximum value of $\espr(G)$ over bipartite $n$-vertex graphs $G$ with diameter 3 is $\frac18n^4-O(n^2)$. In Section \ref{peri}, we detail computations that establish the maximum value of the peripherality index over $n$-vertex graphs and over $n$-vertex trees, $n\leq 8$.

In Section \ref{NTS}, we disprove two conjectures from Furtula \cite{NT} about the Trinajstić index. One conjecture suggests that every graph with a Trinajstić index of $0$ must be regular. We provide several infinite classes of counterexamples to this conjecture. The other conjecture guesses that a certain family of graphs achieves the maximum possible value of the Trinajstić index of an $n$-vertex graph. We show that this conjecture has infinitely many counterexamples by determining that the maximum value achieved by the Trinajstić index of an $n$-vertex graph is $(0.5-o(1))n^4$ and proving that the family described does not achieve this asymptotic maximum.


\section{Preliminaries}

All graphs discussed are simple finite undirected graphs unless otherwise specified. The set of vertices and the set of edges of a graph $G$ will be denoted $V$ and $E$, respectively. For vertices $u,v$ of graph $G$, $d(u,v)$ is defined to be the graph theoretic distance between $u$ and $v$. For vertices $u,v$ of graph $G$, $n_G(u,v)$ is defined to be the number of vertices $x$ of $G$ such that $d(x,u)<d(x,v)$. In other words, it is the number of vertices that are closer to $u$ than to $v$. Naturally, a vertex $u$ for which $n_G(u,v)$ tends to be relatively small can be considered to be peripheral, and a vertex $u$ for which $n_G(u,v)$ tends to be relatively large can be considered to be central. The peripherality of a vertex, denoted $\peri(v)$, is the number of vertices $u$ such that $n_G(u,v)>n_G(v,u)$. The peripherality of a graph is defined as $\peri(G)=\sum\limits_{v\in V}\peri(v)$.

The edge peripherality of an edge, denoted $\eperi(\{u,v\})$, is the number of vertices $x$ such that $n_G(x,u)>n_G(u,x)$ and $n_G(x,v)>n_G(v,x)$. In other words, the vertices $x$ counted are the ones for which more vertices are closer to $x$ than to $u$ and more vertices are closer to $x$ than to $v$. The edge peripherality of a graph is defined as $\eperi(G)=\sum\limits_{\{u,v\}\in E}\eperi(\{u,v\})$. The edge sum peripherality of an edge is defined as $\espr(\{u,v\})=\sum\limits_{x\in V-\{u,v\}}(n_G(x,u)+n_G(x,v))$. As with edge peripherality, the edge sum peripherality of a graph is defined as $\espr(G)=\sum\limits_{\{u,v\}\in E}\espr(\{u,v\})$.

The Mostar index of a pair of vertices is defined as $\mo(\{u,v\})=|n_G(u,v)-n_G(v,u)|$. Note that the order of $u$ and $v$ does not matter. The Mostar index of a graph is the sum of this quantity over all of its edges, $\mo(G)=\sum\limits_{\{u,v\}\in E}\mo(\{u,v\})$. The total Mostar index sums over all pairs of vertices, $\mo^*(G)=\sum\limits_{\{u,v\}\subset V}\mo(\{u,v\})$.

The Trinajstić index of a pair of vertices is defined as $NT(\{u,v\})=(n_G(u,v)-n_G(v,u))^2$. Note that $\{u,v\}$ need not be an edge and that the order of $u$ and $v$ does not matter. This is similar to the definition of the total Mostar index, except that a square, not an absolute value, is taken. As with the total Mostar index, the Trinajstić index of a graph is given by $NT(G)=\sum\limits_{\{u,v\}\subset V}NT(\{u,v\})$.

The irregularity of an edge is defined as $\irr(\{u,v\})=|\deg(u)-\deg(v)|$. Note that the order of $u$ and $v$ does not matter. As with $\mo(G)$, the irregularity of a graph is given by $\irr(G)=\sum\limits_{\{u,v\}\in E}\irr(\{u,v\})$. The definition of irregularity is similar to the definition of the Mostar index, except that each $n_G$ has been replaced with the degree of a vertex.

\section{Edge Peripherality Index}\label{eperi}

In this section, we prove extremal results about the edge peripherality index. We first define a family of graphs that achieves $\eperi(G)\geq\frac{\sqrt{3}}{24}n^3(1-o(1))$, up from the $\frac2{125}n^3$ achieved by the construction in \cite{gt22}. We also improve the upper bound on $\eperi(G)$ from the $\frac12n^3$ given in \cite{gt22} to $\frac16n^3$.

\begin{definition}
In a graph $G$, call an ordered pair $(\{u,v\},x)$ consisting of an edge $\{u,v\}$ and a vertex $x$ $\mathrm{dominant}$ if
\[n_G(x,u)>n_G(u,x)\textrm{ and }n_G(x,v)>n_G(v,x).\]
\end{definition}

Thus, $\eperi(G)$ is equal to the number of dominant pairs.

\begin{thm}
The maximum possible value of $\eperi(G)$ among all connected graphs of order $G$ is at least $\frac{\sqrt{3}}{24}n^3(1-o(1))$.
\end{thm}

\begin{proof}
Define the constant $\alpha=\frac{\sqrt3-1}2$. For $s$ going to $\infty$, construct graph $G_s$ as follows.

Start with the following complete graphs:
\begin{itemize}
\item $K_{\left\lfloor\frac s{1-\alpha}\right\rfloor}$ with vertices $a_{-1,1},a_{-1,2},\dots,a_{-1,{\left\lfloor\frac s{1-\alpha}\right\rfloor}}$

\item $K_{\left\lfloor s\alpha^i\right\rfloor}$ with vertices $a_{i,1},a_{i,2},\dots,a_{i,{\left\lfloor s\alpha^i\right\rfloor}}$ for each $i\geq 0$ for which $s\alpha^i\geq1$. Let $i_\mathrm{max}$ be the maximum such $i$.

\end{itemize}
Next, connect each complete graph to the same central vertex $v$ through chains of vertices such that larger connected components have longer chains. In particular, for each $-1\leq i\leq i_\mathrm{max}$, construct vertices $b_{i,1},b_{i,2},\dots b_{i,2i_\mathrm{max}+1-2i}$. Add an edge between $b_{i,j}$ and $b_{i,j+1}$ for each $1\leq j\leq 2i_\mathrm{max}-2i$. Then add an edge between $a_{i,1}$ and $b_{i,1}$ and an edge between $b_{i,2i_\mathrm{max}+1-2i}$ and $v$.

We now approximate $\frac{\mathrm{eperi}(G)}{n^3}$. Note that $n=(\frac s{1-\alpha}+s\alpha+s\alpha^2+\cdots)(1\pm o(1))=\frac{2s}{1-\alpha}(1\pm o(1))$ because the vertices $a_{i,j}$ make up all but a negligible portion of the vertices as $s$ goes to infinity. Also, vertex $v$ is placed in such a central position that $d(x,v)<d(y,v)\to n_{G_s}(x,y)>n_{G_s}(y,x)$ for all vertices $x,y$. It is important that the largest ``arm" extending from $v$ (the one containing the vertices $a_{-1,j}$) has less than half of the total vertices in the graph. This is true for sufficiently large $s$.

For every edge $\{u,v\}$ in the complete subgraph consisting of $a_{i,j}$ and for every vertex $x=a_{I,j}$ for $I>i$, we have that $(\{u,v\},x)$ is dominant. When $i=-1$, the edge-vertex pairs described above contribute about $\frac{1}{16}n^3$ dominant pairs. This is because the edge must be chosen from a complete subgraph with about $\frac n2$ vertices, and the vertex must be chosen from about $\frac n2$ other vertices.

When $i\geq 0$, the edge-vertex pairs described above contribute $\frac12(\lfloor s\alpha^i\rfloor)^2\lfloor s\alpha^I\rfloor$ dominant pairs. We now sum over $i$ and $I$, ignoring the floors function, as the error caused is a factor of $(1-o(1))$:
\begin{align*}
\sum_{i=0}^\infty\sum_{I=i+1}^\infty\frac12(\lfloor s\alpha^i\rfloor)^2\lfloor s\alpha^I\rfloor&=\left[\sum_{i=0}^\infty\sum_{I=i+1}^\infty\frac12(s\alpha^i)^2(s\alpha^I)\right](1-o(1))\\
&=\left[\frac12s^3\sum_{i=0}^\infty\sum_{j=1}^\infty\alpha^{3i+j}\right](1-o(1))\\
&=\left[\frac12s^3\sum_{i=0}^\infty\frac{\alpha^{3i+1}}{1-\alpha}\right](1-o(1))\\
&=\left[\frac12s^3\frac{\alpha}{(1-\alpha)(1-\alpha^3)}\right](1-o(1))\\
&=\left[\left(\frac{\sqrt{3}}{24}-\frac1{16}\right)n^3\right](1-o(1))
\end{align*}
Adding to the previous $\frac1{16}n^3$, we get that $\mathrm{eperi}(G)\geq\frac{\sqrt3}{24}n^3(1-o(1))$, proving the theorem.
\end{proof}

The above result is an improvement from the family of graphs provided in Theorem 8.11 of \cite{gt22} that achieved $\eperi(G)=\frac2{125}n^3$. Next, we improve the upper bound on $\eperi(G)$.

\begin{thm}
The maximum possible value of $\eperi(G)$ among all connected graphs of order $G$ is at most $\frac16n^3$.
\end{thm}

\begin{proof}
Let $G$ be a connected graph with $n$ vertices.

For any triple $(v_1,v_2,v_3)$ of vertices, at most one of the following is a dominant pair:
\begin{itemize}
\item $(\{v_1,v_2\},v_3)$
\item $(\{v_2,v_3\},v_1)$
\item $(\{v_3,v_1\},v_2)$
\end{itemize}

For example, if the first two are both dominant pairs, then $n_G(v_1,v_3)>n_G(v_3,v_1)>n_G(v_1,v_3)$, a contradiction. Since there are only $\binom n3<\frac16n^3$ triples of vertices, it follows that the number of dominant pairs is $\mathrm{eperi}(G)<\frac16n^3$.
\end{proof}

The above result is an improvement from the $\frac12 n^3$ upper bound proven in Theorem 8.11 of \cite{gt22}.

\section{Edge Sum Peripherality Index}\label{espr}

In this section, we prove extremal results about the edge sum peripherality index. We first define a family of graphs that achieves $\espr(G)\geq\frac5{32}n^4-O(n^3)$, up from the $\frac18n^4-O(n^2)$ achieved by the construction in \cite{gt22}. We also improve the upper bound on $\espr(G)$ from $n^4$ to $\frac14n^4$. Furthermore, we prove that the maximum value of $\espr(G)$ over graphs $G$ of diameter $2$ is $\frac4{27}n^4-O(n^3)$. Additionally, we find that the maximum value of $\espr(G)$ over bipartite graphs $G$ of diameter at most $3$ is $\frac18n^4-O(n^2)$.

\subsection{Bounds on maximum edge sum peripherality over $n$-vertex graphs}

In Corollary 6.6 of \cite{gt22}, the authors supplied a family of graphs achieving $\espr(G)=\frac18n^4-O(n^2)$. We improve this lower bound through a family of graphs achieving $\espr(G)=\frac5{32}n^4-O(n^3)$. We begin with a helpful lemma.

\begin{lemma}
The index $\espr(G)$ can be approximated as
\[\espr(G)=\sum\limits_{u\in V}\deg(u)\sum\limits_{x\in V}n_G(x,u)-O(n^3).\]
\end{lemma}
\begin{proof}
The approximation follows from the following computation.
\begin{align*}
\espr(G)&=\sum_{\{u,v\}\in E}\sum_{x\in V-\{u,v\}}n_G(x,u)+n_G(x,v)\\
&=\sum_{\{u,v\}\in E}\sum_{x\in V}n_G(x,u)+n_G(x,v)-O(n^3)\\
&=2\sum_{\{u,v\}\in E}\sum_{x\in V}n_G(x,u)-O(n^3)\\
&=\sum_{u\in V}\deg(u)\sum_{x\in V}n_G(x,u)-O(n^3),
\end{align*}
where the factor of $2$ is eliminated in the last step to account for double-counting edges.
\end{proof}

\begin{thm}
The maximum possible value of $\espr(G)$ among all connected graphs of order $G$ is at least $\frac5{32}n^4-O(n^3)$.
\end{thm}

\begin{proof}
Let $n=4s+1$. Define vertices $a_i,b_i,c_i,d_i$ for each $1\leq i\leq s$. Also define vertex $v$.
\begin{itemize}
\item Add an edge between $a_i$ and $a_j$ for each $i\neq j$; do the same for $b,c,$ and $d$.
\item Add an edge between $a_i$ and $b_j$ for each $i$ and $j$.
\item Add an edge between $c_i$ and $d_j$ for each $i$ and $j$.
\item Add an edge between $v$ and $b_i$ for each $i$; do the same for $c$.
\end{itemize}
We now proceed to compute $\espr(G)$. By Lemma 3.1, we want to compute $\sum\limits_{u\in E}\deg(u)\sum\limits_{x\in V}n_G(x,u)$. For any $1\leq i,j\leq s$:
\begin{itemize}
\item If $u=a_i$ and $x=b_j$, $x=c_j$, or $x=d_j$, then $\deg(u)\geq 2s$ and $n_G(x,u)\geq 2s$.
\item If $u=b_i$ and $x=c_j$ or $x=d_j$, then $\deg(u)\geq 2s$ and $n_G(x,u)\geq 2s$.
\item If $u=c_i$ and $x=a_j$ or $x=b_j$, then $\deg(u)\geq 2s$ and $n_G(x,u)\geq 2s$.
\item If $u=d_i$ and $x=a_j$, $x=b_j$, or $x=c_j$, then $\deg(u)\geq 2s$ and $n_G(x,u)\geq 2s$.
\end{itemize}
Summing over the four cases above, which constitute all but a negligible portion of the entire sum, we have $\mathrm{espr}(G)\geq 40s^4-O(n^3)=\frac5{32}n^4-O(n^3)$, proving the theorem.
\end{proof}

This construction is an improvement from the one given in Theorem 6.6, which achieved $\espr(G)=\frac18n^4-O(n^2)$. We now turn to improve the upper bound.

\begin{thm}
The maximum possible value of $\espr(G)$ among all connected graphs of order $G$ is at most $\frac14n^4$.
\end{thm}

\begin{proof}
Let $G$ be any connected graph with $n$ vertices. Again by Lemma 3.1,
\begin{align*}
\mathrm{espr}(G)&=\sum\limits_{u\in V}\deg(u)\sum\limits_{x\in V}n_G(x,u)-O(n^3)\\
&\leq\sum\limits_{u\in V}\deg(u)\sum\limits_{x\in V}(n-\deg(u))-O(n^3)\\
&=\sum\limits_{u\in V}\deg(u)(n)(n-\deg(u))-O(n^3)\\
&\leq\sum\limits_{u\in V}\frac14n^3-O(n^3)\\
&=\frac14n^4-O(n^3)
\end{align*}
\end{proof}

The above result is an improvement from the $n^4$ upper bound on $\espr(G)$ given in Corollary 6.7 of \cite{gt22}.

\subsection{Diameter-specific extremal results}\label{esprdiam2}

\begin{thm}
The maximum value of $\mathrm{espr}(G)$ over all $n$-vertex graphs $G$ of diameter $2$ is $\frac4{27}n^4-O(n^3)$.
\end{thm}

\begin{proof}
Note that $2\sum\limits_{\{u,x\}\in E}\deg(u)\deg(x)$ is the number of walks of length $3$ (counted by the two middle vertices). This is known to be greater than or equal to $\frac{8|E|^3}{n^2}$, a fact that is used later in the proof.

Let $G$ be a graph with diameter less than or equal to $2$. We first prove the upper bound.
\begin{align*}
\espr(G)&=\sum\limits_{u\in V}\deg(u)\sum\limits_{w\in V}n_G(w,u)-O(n^3)\\
&=\sum\limits_{u\in V}\deg(u)\sum\limits_{w\in V}\sum\limits_{x\in V}\mathbbm1[d(x,w)<d(x,u)]-O(n^3)\\
&=\sum\limits_{u\in V}\deg(u)\sum\limits_{x\in V}\begin{cases}0&\{u,x\}\in E\\\deg(x)&\text{else}\end{cases}-O(n^3)\\
&=\sum\limits_{u\in V}\sum\limits_{x\in V}\begin{cases}0&\{u,x\}\in E\\\deg(u)\deg(x)&\text{else}\end{cases}-O(n^3)\\
&=\sum\limits_{u\in V}\sum\limits_{x\in V}\deg(u)\deg(x)-2\sum\limits_{\{u,v\}\in E}\deg(u)\deg(x)-O(n^3)\\
&\leq\sum\limits_{u\in V}\sum\limits_{x\in V}\deg(u)\deg(x)-\frac{8|E|^3}{n^2}-O(n^3)\\
&=\left(\sum\limits_{u\in V}\deg(u)\right)\left(\sum\limits_{x\in V}\deg(x)\right)-\frac{8|E|^3}{n^2}-O(n^3)\\
&=4|E|^2-\frac{8|E|^3}{n^2}-O(n^3)\\
&\leq\frac4{27}n^4-O(n^3).
\end{align*}
We now prove the lower bound of $\frac4{27}n^4-O(n^3)$.\\\\
Consider any graph, $G$, with $n$ vertices such that $\deg(v)\approx\frac23n$ for each vertex $v$. We will show that $\mathrm{espr}(G)\approx\frac{4}{27}n^4-O(n^3)$.
\begin{align*}
\mathrm{espr}(G)&=\sum\limits_{u\in V}\deg(u)\sum\limits_{x\in V}n_G(x,u)-O(n^3)\\
&=\sum\limits_{u\in V}\deg(u)\sum\limits_{x\in V}\begin{cases}0&\{u,x\}\in E\\\deg(x)&\text{else}\end{cases}-O(n^3)\\
&=\sum\limits_{u\in V}\frac23n\cdot\frac13n\cdot\frac23n-O(n^3)\\
&=\frac4{27}n^4-O(n^3).
\end{align*}
Thus, the upper and lower bounds have been proven, and the maximum value of $\mathrm{espr}(G)$ over all $n$-vertex graphs $G$ of diameter $2$ is indeed $\frac4{27}n^4-O(n^3)$.
\end{proof}

\begin{thm}
The maximum value of $\mathrm{espr}(G)$ over all bipartite $n$-vertex graphs $G$ of diameter at most $3$ is $\frac18n^4-O(n^2)$.
\end{thm}

\begin{proof}
The lower bound is given by the complete bipartite graph with the maximum number of edges (which has diameter 2). By Corollary 6.6 of \cite{gt22}, this graph has $\espr(G)=\left\lfloor\frac{n^2}4\right\rfloor\left(2\left\lfloor\frac{n^2}4\right\rfloor-2\right)$, which is $\frac18n^4-O(n^2)$.

Now, the following computation proves an upper bound on $\espr(G)$ over bipartite graphs $G$ with diameter at most 3. Below, $V_u$ denotes the (vertex set of the) side of the bipartite graph containing $u$, and $V_u'$ denotes the other side.

\begin{align*}
\espr(G)&\leq\sum\limits_{u\in V}\deg(u)\sum\limits_{x\in V}n_G(x,u)\\
&=\sum\limits_{u\in V}\deg(u)\sum\limits_{x\in V}\sum\limits_{y\in V}\mathbbm1[d(y,x)<d(y,u)]\\
&\leq\sum\limits_{u\in V}\deg(u)[|E|+(|V_u'|-\deg(u))n]\\
&=|E|\sum\limits_{u\in V}\deg(u)+n\sum\limits_{u\in V}\deg(u)(|V_u'|-\deg(u))\\
&=2|E|^2+n\left[\sum\limits_{u\in V_1}\deg(u)(|V_2|\deg(u))+\sum\limits_{u\in V_2}\deg(u)(|V_1|\deg(u))\right]\\
&\leq2|E|^2+n\left[\sum\limits_{u\in V_1}\frac{|E|}{|V_1|}(|V_2|\frac{|E|}{|V_1|})+\sum\limits_{u\in V_2}\frac{|E|}{|V_2|}(|V_1|\frac{|E|}{|V_2|})\right]
\end{align*}
\begin{align*}
&=2|E|^2+n\left[|V_1|\cdot\frac{|E|}{|V_1|}(|V_2|\frac{|E|}{|V_1|})+|V_2|\cdot\frac{|E|}{|V_2|}(|V_1|\frac{|E|}{|V_2|})\right]\\
&=2|E|^2+n|E|(|V_2|-\frac{|E|}{|V_1|}+|V_1|-\frac{|E|}{|V_2|})\\
&=2|E|^2+n|E|(n-\frac{n|E|}{|V_1||V_2|})\\
&\leq 2|E|^2+n|E|(n-\frac{n|E|}{n^2/4})\\
&=|E|(n^2-2|E|)\\
&\leq\frac{n^4}8.
\end{align*}

The first line is by Lemma 3.1. The sixth line is by Jensen's inequality. The third line is due to the following work.

To upper bound the quantity $\sum\limits_{x\in V}\sum\limits_{y\in V}\mathbbm1[d(y,x)<d(y,u)]$, we swap the summations and do casework on $y$.

Case 1: $d(y,u)=2$. Then $d(y,x)=1$. Since $y$ must be in $V_u$, every edge $\{x,y\}$ is counted at most once. Thus, this case contributes at most $|E|$ to the summation.

Case 2: $d(y,u)=3$. Then $y$ is in $V_u'$ but is not a neighbor of $u$. After choosing $y$, there are at most $n$ choices for $x$. Thus, this case contributes at most $(|V_u'|-\deg(u))n$.

This justifies the third line of the computation, $\sum\limits_{x\in V}\sum\limits_{y\in V}\mathbbm1[d(y,x)<d(y,u)]\leq |E|+(|V_u'|-\deg(u))n$.

\end{proof}

Although this proof is specific to bipartite graphs of diameter at most 3, we conjecture that the bound holds for all bipartite graphs.

\begin{conjecture}
The maximum value of $\espr(G)$ over all bipartite $n$-vertex graphs $G$ is $\frac18 n^4-O(n^2)$.
\end{conjecture}

\section{Peripherality Index}\label{peri}

We compute the exact value of the maximum peripherality index over connected $n$-vertex graphs and over $n$-vertex trees for $n\leq 8$. The authors of \cite{gt22} prove that this maximum is given by $\binom n2$ for $n\geq 9$ (for both trees an connected graphs).

We first establish the following lemma.

\begin{lemma}
Any $n$-vertex graph, $G$, with a nontrivial automorphism has $\peri(G)<\binom n2$.
\end{lemma}

\begin{proof}
Let $u$ be a vertex that is mapped to a different vertex, $v$, by some nontrivial automorphism on $\text{spr}(G)$.\\\\
Then $n_G(u,v)=n_G(v,u)$. It follows that $\peri(G)=\sum\limits_{u\neq v}\mathbbm1[n_G(u,v)\neq n_G(v,u)]<\binom n2$, proving the lemma.
\end{proof}

\begin{thm}
The maximum value of the peripherality index over connected $n$-vertex graphs and over $n$-vertex trees for $n\leq 8$ is given by the following table:
\end{thm}
\begin{table*}[ht]
\centering
\begin{tabular}{ |c|c|c| } 
 \hline
 $n$ & $\max\peri(T)$ & $\max\peri(G)$ \\ 
 \hline
1 & 0 & 0 \\ 
2 & 0 & 0 \\
3 & 2 & 2 \\
4 & 4 & 5 \\
5 & 9 & 9 \\
6 & 13 & 15 \\
7 & 21 & 21 \\
8 & 27 & 28 \\
 \hline
\end{tabular}
\caption{The maximum value of $\peri(T)$ and $\peri(G)$ over $n$-vertex trees $T$ and connected $n$-vertex graphs, $G$, for $1\leq n\leq 8$.}
\label{table:1}
\end{table*}

\begin{proof}
For $n=1$ and $n=2$, the only allowed graphs have a peripherality index of $0$.

For $n=3$, the maximum in both cases is achieved by the three-vertex tree.

For $n=4$, the maximum over trees is achieved by the $4$-vertex path, and the maximum over connected graphs is given by a triangle with a pendent vertex.

For $n=5$, note that any $5$-vertex graph, $G$, has a nontrivial automorphism. Thus, by the Lemma, $\peri(G)\leq\binom52-1=9$. Equality is achieved by the spider with legs of length 1, 1, and 2.

For $n=6$ over connected graphs, $\peri(G)=\binom62=15$ is achieved by the graph with vertices $1,2,\dots,6$ and edges $\{\{1,2\}, \{2,3\}, \{3,4\}, \{4,5\}, \{5,6\}, \{3,5\}\}$. For $n=6$ over trees, every tree has a nontrivial automorphism. Of these trees, if some automorphism maps $u_1$ to $v_1$ and $u_2$ to $v_2$ for distinct $u_1,u_2,v_1,v_2$, then we have $n_G(u_1,v_1)=n_G(v_1,u_1)$ and $n_G(u_2,v_2)=n_G(v_2,u_2)$, so by similar reasoning to that used to prove the Lemma, $\peri(T)\leq\binom62-2=13$. The only tree that is not taken into account by the above argument is the spider with legs of length 1, 1, and 3, which has an index of $12$. $\spr(T)=13$ is achieved by the spider with legs of length 1, 2, and 2.

For $n=7$, the balanced spider with legs of length 1, 2, and 3 achieves $\peri(G)=\binom72=21$.

For $n=8$ over connected graphs, $\peri(G)=\binom82=28$ is achieved by the graph with vertices $1,2,\dots,8$ and edges $\{\{1,2\}, \{2,3\}, \{3,4\}, \{4,5\}, \{5,6\}, \{6,7\}, \{7,8\}, \{4,6\}\}$. For $n=8$ over trees, the spider with legs of length 1, 1, 2, and 3 achieves $\peri(T)=27$. To see why $\peri(T)=\binom72=28$ is not achievable, by the Lemma, it suffices to check trees with no nontrivial automorphism. The only such tree is the spider with legs of length 1, 2, and 4, which has $\peri(T)<28$.
\end{proof}

\section{Trinajstić Index}\label{NTS}

Conjecture 3.2 of \cite{NT} suggests that every graph with a Trinajstić index of 0 must be regular. Conjecture 3.3 suggests that the maximal value of the Trinajstić index over all $n$-vertex graphs $G$ is achieved by ``a complete subgraph $K_{\lceil n/2\rceil}$, on whose vertices are attached $\lfloor n/2\rfloor$ pendent vertices." Both conjectures were based on a computer search of all graphs with $5$ to $10$ vertices. In this section, we supply a small counterexample to Conjecture 3.3 and prove that the conjecture is also false for infinitely many large values of $n$. In order to do so, we first show that the maximum value of $NT(G)$ over all graphs $G$ is $(0.5-o(1))n^4$. Additionally, we present a class of counterexamples to Conjecture 3.2.

Conjecture 3.3 of \cite{NT} states the following:

\begin{conjecture}
The graph with the maximal value of the Trinajstic index is unique. It is consisting of a complete subgraph $K_{\lceil n/2\rceil}$, on whose vertices are attached $\lfloor n/2\rfloor$ pendent vertices. One pendent vertex is allowed per vertex that belongs to the complete subgraph. \cite{NT}
\end{conjecture}

\begin{thm}
Conjecture 5.1 is false for $n=4$.
\end{thm}
\begin{proof}
For $n=4$, the graph described in Conjecture 3.3 is the path $P_4$, which has a Trinajstić index of 10. The star graph $S_3$ has a Trinajstić index of $12>10$, disproving the conjecture.
\end{proof}

However, the conjecture is not just false for one value of $n$; it is false for infinitely many, because the graph that it describes does not achieve the asymptotic maximum of $(0.5-o(1))n^4$ which is proven below.

\begin{thm}
The maximum value of $NT(G)$ over all graphs $G$ is $(0.5-o(1))n^4$.
\end{thm}
\begin{proof}
We first show the upper bound. Each of the $\binom n2$ pairs of vertices contributes at most $(n-2)^2$ to $NT(G)$, giving an upper bound of $\binom n2(n-2)^2=(0.5-o(1))n^4$.

To show the lower bound, we consider the Trinajstić index of balanced spider graphs. As in \cite{gt22}, let $S_{a,b}$ denote the balanced spider graph with $a$ legs of length $b$, each attached to a central vertex (for a total of $ab+1$ vertices).

For any vertices $u,v$ with $d(u,c)\neq d(v,c)$ (where vertex $c$ is the center of $S_{a,b}$), we have $(n_G(u,v)-n_G(v,u))^2\geq (n-2b)^2$.

There are at least $\frac{n(n-a)}2$ such pairs $u,v$, so
\begin{align*}
NT(S_a,b)&\geq\frac{n(n-a)}2(n-2b)^2\\
&\geq \frac{n^2(1-\frac 1b)}2n^2\left(1-\frac 2a\right)^2\\
&\geq n^4(0.5)\left(1-\frac1b-\frac4a\right)
\end{align*}
The second inequality uses the fact that $n=ab+1$, and the third uses the inequality $(1-x)(1-y)(1-z)\geq1-x-y-z$ with $x=\frac1b$ and $y=z=\frac2a$.

Choosing $a$ and $b$ to be arbitrarily large, we have that $NT(S_{a,b})=(0.5-o(1))n^4$, proving the lower bound.
\end{proof}

\begin{thm}
Conjecture 5.1 is false for infinitely many $n$.
\end{thm}
\begin{proof}
Consider the graph, $G$, described in the conjecture: ``a complete subgraph $K_{\lceil n/2\rceil}$, on whose vertices are attached $\lfloor n/2\rfloor$ pendent vertices. One pendent vertex is allowed per vertex that belongs to the complete subgraph." 

Choosing $u$ from the complete subgraph and $v$ from the set of pendent vertices makes $(n_G(u,v)-n_G(v,u))^2$ approximately $n^2$. All other choices of $\{u,v\}$ contribute negligibly, if at all, to $NT(G)$. Thus, $NT(G)\approx 0.25n^4$ for the graph $G$ described in the conjecture. By Theorem 5.2, there exist arbitrarily large graphs that have a greater Trinajstić index than the graph given by Conjecture 5.1.
\end{proof}

Conjecture 3.2 of \cite{NT} suggests the following:
\begin{conjecture}
The minimum value of the Trinajstic index is equal to 0. The necessary but not sufficient condition for a graph to reach the minimum of the Trinajstic index is to be a regular graph.
\end{conjecture}
\noindent
In the proof of the following theorem, we provide a class of counterexamples to Conjecture 5.2.
\begin{thm}
There exist arbitrarily large graphs that are not regular but that achieve the minimum Trinajstić index of 0.
\end{thm}
\begin{proof}
For any integer $a$, let $n_{a}(u,v)$ be the number of vertices, $x$, such that $d(x,u)<a+d(x,v)$. Let $N_{a}(u,v)=n_{a}(u,v)-n_{a}(v,u)$. Note that $NT(G)=0$ if $N_0(u,v)=0$ for all pairs $u,v$ of vertices. We will call such a graph NT-balanced. Call a graph ultra NT-balanced if $N_a(u,v)=0$ for all integers $a$ and vertices $u,v$.

Obviously, every ultra NT-balanced graph is also NT-balanced. Furthermore, the graphs of the rhombic dodecahedron and rhombic triacontahedron are both ultra NT-balanced. Neither graph is regular. In order to generate arbitrarily large counterexamples as promised, we prove the following claim.
\begin{claim}
If $G$ and $H$ are ultra NT-balanced, then $G\square H$ is as well.
\end{claim}
\begin{proof}

Fix integer $a$ and vertices $u=z_{i_1,j_1}$ and $v=z_{i_2,j_2}$. For $(a,b)\in\{1,2,\dots,m\}^2$, let $s(a,b)=d(x_a,x_b)$ in $G$. For $(c,d)\in\{1,2,\dots,n\}^2$, let $t(c,d)=d(y_c,y_d)$ in $H$. Obviously, $d(z_{a,c},z_{b,d})=s(a,b)+t(c,d)$ in $G\square H$.

Since $G$ is ultra NT-balanced, it follows that there exists an involution $X$ on $\{1,2,\cdots,m\}$ such that $s(i,i_1)-s(i,i_2)=-s(X(i),i_1)+s(X(i),i_2)$.

Define an analogous involution $Y$ on $\{1,2,\cdots,n\}$.

Then $Z((i,j))=(X(i),Y(j))$ is an involution on $\{1,2,\dots,m\}\times\{1,2,\dots,n\}$ with the property that
\begin{align*}
d(z_{i,j},z_{i_1,j_1})-d(z_{i,j},z_{i_2,j_2})&=(s(i,i_1)+t(j,j_1))-(s(i,i_2)+t(j,j_2))\\
&=(s(i,i_1)-s(i,i_2))+(t(j,j_1)-t(j,j_2))\\
&=(-s(X(i),i_1)+s(X(i),i_2))+(-t(Y(j),j_1)+t(Y(j),j_2))\\
&=-(s(X(i),i_1)+t(Y(j),j_1))+(s(X(i),i_2)+t(Y(j),j_2))\\
&=-d(z_{X(i),Y(j)},z_{i_1,j_1})+d(z_{X(i),Y(j)},z_{i_2,j_2}).
\end{align*}
Therefore, this involution pairs vertices $z_{i,j}$ that are counted in $n_a(u,v)$ with vertices $z_{X(i),Y(j)}$ that are counted in $n_a(v,u)$, proving that $n_a(u,v)=n_a(v,u)$. Thus, $G\square H$ is ultra NT-balanced, as claimed.
\end{proof}

The promised arbitrarily large follow easily from the claim, for example by crossing the rhombic dodecahedron graph with $K_n$ or $C_n$ for arbitrary $n$.
\end{proof}
\begin{remark}
Every NT-balanced graph whose existence is implied by the proof of Theorem 5.3 is also Ultra NT-balanced. Thus, we set forth the following conjecture.
\end{remark}
\begin{conjecture}
Every NT-balanced graph is also Ultra NT-balanced.
\end{conjecture}

\section{Discussion and future directions}

In this paper, we improved the bounds on the maximum value achieved by $\eperi$ and $\espr$ over $n$-vertex graphs and over other classes of graphs. We computed the maximum of $\peri$ over $n$-vertex trees and graphs for $n\leq8$, completing a result from \cite{gt22} that does the same for $n\geq9$. We disproved two conjectures from \cite{NT} about graphs that maximize and minimize the Trinajstić index.

One direction for future research is to continue improving bounds on the maxima of $\eperi,\espr,$ and $NT$ over connected graphs. The corresponding maxima of these measures over bipartite graphs also remain open.

Another possible direction is the investigation of the expected value of peripherality indices of random graphs. For example, \cite{gt22} proves the following result:
\begin{thm}
The expected value of $\irr(G_{n,\frac12})$ is $\frac{n^{5/2}(1-o(1))}{4\sqrt\pi}$.
\end{thm}
This result can be generalized to any probability $p$ in place of $\frac12$.
\begin{thm}
The expected value of $\text{irr}(G_{n,p})$ is $=p\sqrt{\frac{p(1-p)}\pi}n^{5/2}(1\pm o(1))$
\end{thm}
\begin{proof}
By linearity of expectation, suffices to find the contribution from each unordered pair $(u,v)$ and multiply it by $\binom n2$. Let $u$ and $v$ be the two vertices. The probability that they are connected by an edge is $p$.

For all vertices $w\not\in\{u,v\}$, let $X_w=\mathbbm1[\{u,w\}\in E]-\mathbbm1[\{v,w\}\in E]$. Thus, $X_w$ is $1$ with probability $p(1-p)$, $-1$ with probability $p(1-p)$, and $0$ otherwise. It follows that the distribution $X_w$ has standard deviation $\sqrt{2p(1-p)}$, and all $n-2$ of them are independent.

By the Central Limit Theorem, $\deg(u)-\deg(v)=\sum\limits_{w\not\in\{u,v\}}X_w$ approximates a normal distribution with standard deviation $\sigma=\sqrt{2p(1-p)(n-2)}$. Let $k$ be the ratio of the mean absolute deviation to the standard deviation of an arbitrary normal distribution. Then the expected value of $|\deg(u)-\deg(v)|$ is approximately $k\sigma$.

Thus, the expected value of $|\text{irr}(G_{n,p})|$ is $\binom n2k\sigma=(1\pm o(1))pk\sqrt{\frac{p(1-p)}2}n^{2.5}$. The only value of $k$ that makes Theorem 6.1 correct is $k=\sqrt{\frac2\pi}$. Plugging it in, the expected value of $|\text{irr}(G_{n,p})|$ is $(1\pm o(1))p\sqrt{\frac{p(1-p)}\pi}n^{2.5}$, as desired.
\end{proof}

Another direction for future research is to attempt to define peripherality measures that harness the full structure of chemical networks. Chemical networks, such as MOZART-4 and SuperFast, consist not of undirected edges, but rather of chemical reactions with several reactants and several products. However, most existing measures, including all of the measures studied in this paper, are specific to undirected graphs.

\section{Acknowledgements}
The author is very grateful to Dr.\ Jesse Geneson for suggesting this research project and guiding the research and the writing of this paper. The author also thanks Dr.\ Tanya Khovanova for her feedback on drafts of this paper. The author also thanks the PRIMES USA program for making this research possible.

\end{document}